\documentclass[12pt,reqno]{amsart}

\usepackage{preamble_paper}
\usepackage{preamble_general}
\usepackage{preamble_math}

\begin{document}

\title[Equidistribution of Diophantine pairs]{Equidistribution of Diophantine pairs among the equivalence classes of quadratic forms}

\author{Goran Dra\v{z}i\'{c}}
\address{Goran Dra\v{z}i\'{c}, Faculty of Food Technology and Biotechnology, University of Zagreb, Croatia}
\email{Goran.Drazic@pbf.unizg.hr}

\author{Matija Kazalicki}
\address{Matija Kazalicki, Faculty of Science, Department of Mathematics, University of Zagreb, Croatia}
\email{Matija.Kazalicki@math.hr}

\author{Rudi Mrazovi\'{c}}
\address{Rudi Mrazovi\'{c}, Faculty of Science, Department of Mathematics, University of Zagreb, Croatia}
\email{Rudi.Mrazovic@math.hr}

\newcommand{\goran}[1]{\authorcomment{Goran}{#1}}
\newcommand{\matija}[1]{\authorcomment{Matija}{#1}}
\newcommand{\rudi}[1]{\authorcomment{Rudi}{#1}}

\begin{abstract}
    For a fixed integer $n$, we say that $\{a,c\}\subset\mathbf{Z}\setminus\{0\}$ is a $D(n)$-pair if $ac+n$ is a perfect square.
    In this short note we prove that $D(n)$-pairs are asymptotically equidistributed (via their associated quadratic forms) among proper $\mathrm{SL}_2(\mathbf{Z})$-equivalence classes of binary quadratic forms of discriminant $4n$ with fixed content. As a consequence, we obtain a more streamlined and simpler proof of Badesa's asymptotic formula for the number of $D(n)$-pairs.
\end{abstract}

\maketitle

\section{Introduction}

Let $R$ be a ring and $n \in R$.
A $D(n)$-$m$-tuple in $R$ is a set of $m$ distinct nonzero elements such that the product of any two of them plus $n$ is a perfect square.
When $R = \Z$ (or $\Q$) and $n=1$, such sets are called (rational) Diophantine $m$-tuples, and they have been studied since antiquity.

The first known rational Diophantine quadruple was found by Diophantus:
\[
    \Bigl\{\frac{1}{16},\, \frac{33}{16},\, \frac{17}{4},\, \frac{105}{16}\Bigr\}.
\]
In the integer setting, Fermat discovered the quadruple $\{1,3,8,120\}$, which Euler extended to the rational quintuple
\[
\left\{1, 3, 8, 120, \frac{777480}{8288641}\right\}.
\]
Stoll~\cite{Stoll} later proved that this is the only such extension.

In 1969 Baker and Davenport~\cite{B-D} showed that $d=120$ is the only integer that completes $\{1,3,8\}$ to a Diophantine quadruple. This led to the conjecture---recently confirmed by He, Togb\'e, and Ziegler \cite{HTZ} (see also \cite{duje-crelle})---that there are no Diophantine quintuple in integers.

For rational tuples, the situation is more complicated. 
In 1999, Gibbs~\cite{Gibbs1} found the first rational sextuple.
Later, Dujella, Kazalicki, Miki\'c, and Szikszai~\cite{DKMS} proved that there are infinitely many such sextuples. 
Dujella and Kazalicki~\cite{Duje-Matija} introduced new parametric constructions,
while Dujella, Kazalicki, and Petri\v{c}evi\'c~\cite{DKP-sext,DKP-reg} found sextuples with square denominators and regular substructures. 
No rational septuple is currently known,
and Lang's conjecture suggests that there should be a uniform bound on the size of such sets.

Beyond integers and rationals, 
Diophantine $m$-tuples have also been studied in other rings.
Dujella and Kazalicki~\cite{DK-finite} computed their number over finite fields.

Dra\v{z}i\'c and Kazalicki~\cite{DrazicKazalicki} studied rational $D(n)$-quadruples with fixed product via elliptic curves. 
It remains open whether rational $D(n)$-quintuples exist for all $n$.
However, assuming the Parity Conjecture for certain families of elliptic curves,
Dra\v{z}i\'c \cite{Drazic} proved that for at least 99.5\% of squarefree integers $n$, 
there exist infinitely many rational $D(n)$-quintuples. 
When $n$ is not a square,
no rational $D(n)$-sextuple is known.

\begin{comment}
Another direction involves tuples valid for multiple values of $n$. Ad\v{z}aga, Dujella, Kreso, and Tadi\'c~\cite{ADKT} constructed infinitely many integer $D(1)$-triples with two additional $D(n)$-properties, and several $D(1)$-triples with three additional $D(n)$-properties. Dujella and Petri\v{c}evi\'c~\cite{DP,DP2} showed that there are infinitely many integer quadruples simultaneously satisfying $D(n_1)$- and $D(n_2)$-properties. Extending these results, Dujella, Kazalicki, and Petri\v{c}evi\'c~\cite{DKP-square} proved the existence of infinitely many essentially different $D(n)$-quintuples composed entirely of square elements.

A rational Diophantine tuple is called strong if each element squared plus one is also a square. Dujella and Petri\v{c}evi\'c~\cite{DujellaPetricEvic2008} showed that infinitely many strong rational triples exist, though no quadruple has been found. For generalizations, see~\cite{DKP-strong}.
\end{comment}

For further details, see the survey article~\cite{Duje-Notices} and the book~\cite{duje-book}.

In this note, we study the asymptotic behaviour, as $T\to\infty$, of the number of $D(n)$-pairs $\{a, b\}$ with $a, b \in \{-T, \ldots, T\}$.
Dujella~\cite{Duje-asymptotic} showed that the number of $D(1)$-pairs grows like $\frac{12}{\pi^2} T \log T$. 
Adžaga, Dražić, Dujella, and Pethő~\cite{ADDP} proved that for $n = -1$ or for $|n|$ prime,
the number of $D(n)$-pairs grows is asymptotic to $kT$,
where $k$ is an explicit constant involving special values of Dirichlet $L$-functions. 
More recently, Badesa~\cite{Badesa} determined the asymptotics for arbitrary integers $n$. 

The common approach in these works is to estimate the number of solutions to the congruence $x^2 \equiv n \pmod{b}$, with $b$ ranging from $1$ to $T$. 
Dirichlet $L$-functions then arise naturally, 
and this makes the constant rather involved to compute---especially for general $n$.

We revisit the problem from a new perspective. 
Our starting point is an observation by Badesa:
to each $D(n)$-pair $\{a, c\}$, one can associate the integral binary quadratic form
\[
    E_{ac} := [a, 2\sqrt{ac+n}, c]
    \qquad \text{with } a > c.
\]
The form $E_{ac}$ lies in a proper equivalence class (that is, equivalent under the $\SL_2(\Z)$-action) of forms with discriminant $4n$.
This naturally leads to the question of how $D(n)$-pairs are distributed among these equivalence classes.

\section{Statement of the main result}

To state our main result, fix an integral binary quadratic form $Q$ of discriminant $4n$, and define
\[
    D^Q_T 
    := \#\bigl\{
        \{a,c\}
        \,\big|\,
        a, c \in \{-T,\ldots,T\},\ \{a,c\} \text{ is a $D(n)$-pair},\ E_{ac} \sim Q
    \bigr\}.
\]
Thus $D_T^Q$ counts those $D(n)$-pairs $\{a,c\}$ whose associated quadratic form $E_{ac}$ is properly equivalent to $Q$.
Our main theorem shows that, for fixed content, $D(n)$-pairs become asymptotically equidistributed among proper equivalence classes of binary quadratic forms.

\begin{theorem}
    \label{thm:main}
    Let $n$ be a nonzero integer,
    and let $Q=[a,b,c]$ be a binary quadratic form of discriminant $4n$ and content $k$;
    equivalently, $Q/k$ is primitive of discriminant $d' = 4n/k^2$.

    \begin{itemize}
        \item[(a)] 
        If $n<0$, then%
        \footnote{As usual, $f(T)=o(g(T))$ means that $f(T)/g(T)\to 0$ as $T\to\infty$.}
        \[
            D_T^Q = \frac{6T}{\omega(d') \pi |n|^{1/2}} + o(T),
        \]
        where
        \[
        \omega(d') = \begin{cases}
            4 &\text{if } d'=-4, \\
            6 &\text{if } d'=-3, \\
            2 &\text{otherwise}.
        \end{cases}
        \]
    
        \item[(b)] 
        If $n>0$ is not a square, then
        \[
            D_T^Q = \frac{12T \log{\epsilon_{d'}}}{\kappa(d')\pi^2n^{1/2}} + o(T), 
        \]
        where $\epsilon_{d'}>1$ is the fundamental unit of the order $\mathcal{O}_{d'} = \Z[\frac{d'+\sqrt{d'}}{2}]$ in the quadratic field $\Q(\sqrt{4n})$.
        Here $\kappa(d')=\frac{h^+(d')}{h(d')}$ is the ratio of the narrow class number and the class number of $\mathcal{O}_{d'}$.

        \item[(c)]
        If $n>0$ is a square, then
        \[
            D_T^Q = \frac{6T \log T}{\pi^2 n^{1/2}} + o(T \log T).
        \]
    \end{itemize}
\end{theorem}

The first, trivial, step in the proof of Theorem \ref{thm:main} is to rephrase everything in the language of quadratic forms.
Indeed, a simple bijection shows that $D_T^Q$ agrees with
\[
    F^Q_T
    :=\# \bigl\{ [a,b,c] \,\big|\, a, c \in \{-T,\ldots,T\}, a > c, b \geq 0, [a,b,c] \sim Q \bigr\}.
\]
With this easy observation in hand,
Theorem \ref{thm:main} follows from the corresponding orbit-counting statement for quadratic forms:

\begin{theorem}
    \label{thm:mainF}
    Under the same assumptions and notation as in Theorem~\ref{thm:main}, all asymptotic formulas in Theorem~\ref{thm:main} remain valid with $D_T^Q$ replaced by $F_T^Q$. In other words, in each of the cases \emph{(a)--(c)} of Theorem~\ref{thm:main}, the quantity $F_T^Q$ satisfies the corresponding asymptotic stated there for $D_T^Q$.
\end{theorem}
We prove Theorem~\ref{thm:mainF} in Sections~\ref{sec:homdyn}--\ref{sec:split}.
For now, we explain how summing over all proper equivalence classes recovers Badesa’s asymptotic formula~\cite{Badesa} (with an alternative expression for its leading coefficient).

\section{A simpler proof of Badesa's count of Diophantine pairs}

In this section we explain how our homogeneous-dynamics approach to Theorem~\ref{thm:main} also yields a substantially simpler (in our view, both conceptually and technically) proof of Badesa’s asymptotic for the number of $D(n)$-pairs. 

In fact, once one is comfortable with the homogeneous-dynamics viewpoint, the argument becomes almost immediate, yet this perspective has not been used before for counting $D(n)$-pairs.
% In fact, once one is comfortable with the homogeneous-dynamics viewpoint, the argument becomes so direct and straightforward that we are somewhat surprised it has not been used earlier in this counting problem.

Recall that Badesa’s proof begins by reducing the count of $D(n)$-pairs $\{a,c\}$ with $|a|,|c|\leq T$ to a congruence-counting problem%
\footnote{As usual, $a_T \sim b_T$ means that $a_T/b_T\to 1$ as $T\to\infty$.}
\[
    D_T^n \sim 2\sum_{t\leq T} S(n,t),
\]
where $S(n,t)$ denotes the number of solutions to $x^{2}\equiv n \pmod t$ (with a minor adjustment for the trivial solutions when $n$ is a square).
He then obtains an asymptotic for $\sum_{t\leq T}S(n,t)$ by analyzing an associated Dirichlet series with Euler product, relating it to $L(s,\chi_n)$, and finally applying a Tauberian theorem to extract the main term.

Obviously, to recover Badesa’s result from Theorem~\ref{thm:main}, it suffices to sum the asymptotic for $D_T^Q$ over all proper equivalence classes of binary quadratic forms of discriminant $4n$. 
We denote the set of these classes by $\cC(4n)$. 
Recall that, for fixed signature and content $k$, these classes are in bijection with the narrow class group $\Cl^+(4n/k^2)$, whose order we denote by $h^+(4n/k^2)$.

In general, for any discriminant $d$ which is not a perfect square, we have
\[
    h^+(d) = 
    \begin{cases}
        h(d), & \text{if } N_{K/\Q}(\epsilon_d) = -1, \\
        2h(d), & \text{if } N_{K/\Q}(\epsilon_d) = +1,
    \end{cases}
\]
where $\epsilon_d$ is a fundamental unit of the order $\cO_d = \Z\bigl[\frac{d+\sqrt{d}}{2}\bigr]$ in the quadratic field $\Q(\sqrt{d})$, and $h(d)$ is the (proper) class number of quadratic forms of discriminant $d$. 
Write $h^+(d)=\kappa(d)h(d)$, where $\kappa(d)\in \{1,2\}$. 
If $d<0$ then $h(d)=h^+(d)$. 
We will also use the following relation between the class numbers of non-maximal orders,
the ring class number formula:
\begin{equation}
    \label{eq:ring}
    h(d_0 f^2)=\frac{h(d_0)f}{[\cO_{d_0}^\times:\cO_{d_0 f^2}^\times]} \prod_{p|f} \Bigl(1-\chi_{d_0}(p)\frac{1}{p}\Bigr),
\end{equation}
where $d_0$ is a fundamental discriminant and $\cO_{d}^\times$ denotes the unit group of the order $\cO_d$.

We are now ready to recover the results of \cite{Duje-asymptotic,ADDP,Badesa}.

\begin{theorem}
    \label{thm:total-count}
    Let $n$ be a nonzero integer.
    \begin{enumerate}
        \item[(a)]
        Assume $n<0$.
        If $4n=f^2 d_0$ with $d_0<0$ a fundamental discriminant, then
        \[
            D^n_T
            = \frac{12T}{\omega(d_0)\pi |n|^{1/2}}h(d_0) \Bigl[  \prod_{p^e || f} \bigl( \sigma_1(p^e) - \chi_{d_0}(p) \sigma_1(p^{e-1}) \bigr)\Bigr]+o(T).
        \]

        \item[(b)]
        Assume $n>0$ is not a square.
        If $4n=f^2 d_0$ with $d_0 > 0$ a fundamental discriminant, then
        \[
            D^n_T
            = \frac{12T \log{\epsilon_{d_0}}h(d_0)}{\pi^2 n^{1/2}} \Bigl[  \prod_{p^e || f} \bigl( \sigma_1(p^e) - \chi_{d_0}(p) \sigma_1(p^{e-1}) \bigr) \Bigr]+o(T),
        \]
        where $\log \epsilon_{d_0}$ is the regulator of the real quadratic field $\Q(\sqrt{d_0})$.
        
        \item[(c)]
        If $n>0$ is a perfect square, 
        then
        \[
            D^n_T
            = \frac{12T \log T}{\pi^2}+o(T \log T).
        \]
    \end{enumerate}
\end{theorem}

In the proof we will use the following simple arithmetic lemma.
Let $\sigma_1(m)$ denote the sum of positive divisors of $m$,
and write $\chi_{d_0}(p)=\bigl( \frac{d_0}{p} \bigr)$ for the Kronecker symbol attached to the fundamental discriminant $d_0$.

\begin{lemma}
    \label{multiplikativnost}
    Let 
    \[
        g(n)=\sum_{l|n} l \prod_{p|l}\Bigl(1-\Bigl(\frac{d_0}{p}\Bigr)\frac{1}{p}\Bigr).
    \]
    Then $g$ is multiplicative and
    \[
        g(p^e) 
        = \sigma_1(p^e)-\sigma_1(p^{e-1}) \Bigl(\frac{d_0}{p}\Bigr).
    \]
\end{lemma}

\begin{proof}
    The function $g_1(l)=\prod_{p|l}(1-(\frac{d_0}{p})\frac{1}{p})$ is multiplicative, 
    and therefore so is $g_2(l)=l\cdot g_1(l)$, being a product of multiplicative functions. 
    Since $g=g_2\ast 1$ is a Dirichlet convolution of multiplicative functions, it is multiplicative as well. For $n=p^e$ we compute
    \[
        g(p^e)
        = 1+\sum_{k=1}^e p^k \Bigl(1-\Bigl(\frac{d_0}{p}\Bigr)\frac{1}{p}\Bigr)
        =\sum_{k=0}^e p^k-\sum_{k=0}^{e-1} p^k\Bigl(\frac{d_0}{p}\Bigr)
        = \sigma_1(p^e)-\sigma_1(p^{e-1})\Bigl(\frac{d_0}{p}\Bigr),
    \]
    as claimed.
\end{proof}

\begin{proof}
    [Proof of Theorem \ref{thm:total-count}]
    We begin with the definite case $n<0$.
    There are two possible signatures (positive and negative definite). 
    Equivalently, for every $D(n)$-pair $\{a,c\}$ the associated form $E_{ac}$ is properly equivalent to either $Q$ or $-Q$, for a positive definite form $Q$ representing the class in $\Cl^+(4n)$ determined by $\{a,c\}$ (throughout the proof we write $Q$ both for a form and for its proper equivalence class).

    By Theorem \ref{thm:main} (a),
    \[
        D^n_T
        = \sum_{Q \in \cC(4n)} D_T^Q
        = 2 \sum_{l |f} \frac{6T}{\omega(d_0 l^2) \pi |n|^{1/2}}h(d_0 l^2)+o(T).
    \]
    Using the ring class number formula~\eqref{eq:ring}, together with the identity $\omega(d_0l^2)[O_{d_0}^\times : O_{d_0 l^2}^\times]=\omega(d_0)$, this becomes
    \[
        D^n_T
        = \frac{12T}{\omega(d_0) \pi |n|^{1/2}}h(d_0)\cdot \sum_{l |f} l \prod_{p|l} \Bigl(1 - \Bigl(\frac{d_0}{p}\Bigr) \frac{1}{p}\Bigr)+o(T).
    \]
    Claim (a) now follows from Lemma \ref{multiplikativnost}.

    Next assume that $n>0$ is not a square. 
    The computation is essentially the same. 
    By Theorem \ref{thm:main} (b),
    \[
        D^n_T
        = \sum_{Q \in \cC(4n)} D_T^Q
        =\sum_{l |f} \frac{12T \log{\epsilon_{d_0 l^2}}}{\kappa(d_0 l^2)\pi^2n^{1/2}} h^+(d_0 l^2) + o(T).
    \]
    Using $\epsilon_{d_0l^2}=\epsilon_{d_0}^{[\cO_{d_0}^\times:\cO_{d_0l^2}^\times]}$ and $h^+=\kappa h$, we obtain
    \[
        D^n_T
        = \frac{12T}{\pi^2 n^{1/2}} \sum_{l|f}[\cO_{d_0}^\times:\cO_{d_0 l^2}^\times] \log{\epsilon_{d_0}} \ h(d_0 l^2)+o(T).
    \]
    Applying the ring class number formula \eqref{eq:ring} gives
    \[
        D^n_T
        = \frac{12T \log{\epsilon_{d_0}}h(d_0)}{\pi^2 n^{1/2}}\sum_{l|f}l\prod_{p|l}\Bigl(1-\chi_{d_0}(p)\frac{1}{p}\Bigr)+o(T),
    \]
    and claim (b) again follows from Lemma \ref{multiplikativnost}.

    Finally, assume that $n=k^2$ is a square. 
    Reduction theory for binary quadratic forms of square discriminant is much simpler than in the nonsquare case.
    Concretely, there are exactly $2k$ proper equivalence classes of integral binary quadratic forms of discriminant $4k^2$ (see, for example, \cite[Exercise 1]{split-zagier}).\footnote{Representatives for these classes are $[0,2k,c]$ for $c=0,1,\dots,2k-1$.}
    The claim follows by summing the asymptotic from Theorem~\ref{thm:main} (c) over these classes.
\end{proof}

\section{$\SL_2(\R)$-action on binary quadratic forms and homogeneous dynamics}
\label{sec:homdyn}

We now turn to the proof of Theorem~\ref{thm:mainF}. 
The argument follows a standard template and uses familiar tools from homogeneous dynamics. 
Since we expect the paper to be of particular interest to researchers coming from the Diophantine-tuples side, 
we include a bit more setup than would typically be required.

To set the stage, 
let $V$ be the $3$-dimensional space of \emph{real} binary quadratic forms. 
As usual, we identify $V$ with the space of symmetric $2\times 2$ real matrices via
\[
    \bigl( Q(x,y) = ax^2 + bxy + cy^2 \bigr)
    \mapsto
    \begin{pmatrix}
        a & b/2 \\
        b/2 & c
    \end{pmatrix}.
\]
Depending on the context, 
we use $[a,b,c]$ and $Q$ to denote either the quadratic form or its associated symmetric matrix. 
In particular, for $v\in\R^2$ we $Q(v) = v^T Q v$.

Let $G=\SL_2(\R)$ and $\Gamma=\SL_2(\Z)$. 
We have already discussed the $\Gamma$-action on integral quadratic forms; 
from now on we consider its extension to a $G$-action on $V$.
In the matrix model, $G$ acts by congruence:
\[
    g\cdot Q := gQg^{T}.
\]

We will study the $G$-orbits of this action. 
By Sylvester’s law of inertia (see, for example, \cite[Section~4.5]{MR2978290}), 
the signature and the discriminant form a complete set of invariants: 
two nondegenerate quadratic forms lie in the same $G$-orbit if and only if they have the same signature and discriminant.

Fix a nondegenerate binary form $Q$, and let $H$ be its stabilizer in $G$.
Then the orbit $G\cdot Q$ is isomorphic (as a homogeneous space\footnote{A $G$-set is called \emph{homogeneous} if the $G$-action is transitive.}) to $G/H$ (with left multiplication by $G$),
and there is a unique (up to scaling) $G$-invariant measure $m_{G/H}$ on $G/H$. 
Via the orbit map, $m_{G/H}$ induces a $G$-invariant measure on $G\cdot Q$.

Given a left-invariant Haar measure $m_H$ on $H$, there exists a left-invariant Haar measure $m_G$ on $G$ such that locally $m_G = m_{G/H} \otimes m_H$,
in the sense that for every $f\in C_c(G)$,
\begin{equation} 
    \label{eq:Fubini}
    \int_G f(g) \,dm_G(g)
    = \int_{G/H} \int_H f(gh) \, dm_H(h) \, dm_{G/H}(gH).
\end{equation}

From now on, fix a nonzero integer $n$ and let $Q$ be an integral binary quadratic form of discriminant $4n$. 
Let
\[
    H = \{g \in G : g Q g^T = Q\}
\]
be the stabilizer of $Q$ for the $G$-action above. 
Note that $H$ is also the fixed-point set of the involution $g \mapsto Q g^{-T} Q^{-1}$, 
a hypothesis needed later when we apply the Eskin--McMullen counting theorem.

Define
\begin{equation}
    \label{eq:cFT}
    \cF_T = \bigl\{ [a,b,c] \in V \,\bigm|\, |a|,|c| \leq T,\ a>c,\ b \geq 0 \bigr\}.
\end{equation}
To prove Theorem~\ref{thm:mainF} we need an asymptotic, as $T\to\infty$, for the number of forms in the $\Gamma$-orbit of $Q$ that lie in $\cF_T$, namely
\begin{equation}
    \label{eq:countingGQ}
    \bigl|(\Gamma \cdot  Q) \cap \cF_T\bigr|.
\end{equation}

To place this into the framework of \cite{MR1230290}, we transfer the problem $G \cdot Q$ to $G/H$. 
Set
\[
    \cB_T 
    := \{ gH \in G/H \mid g \cdot Q \in \cF_T\}.
\]
Then the counting problem becomes
\[
    \bigl|\Gamma\cdot H \cap \cB_T\bigr|,
\]
where we view $H$ as the base point in $G/H$.

A natural heuristic is that
\begin{equation}
    \label{eq:approx-orbit}
    \bigl|\Gamma\cdot H \cap \cB_T\bigr|
    \approx \frac{m_{G/H}(\cB_T)}{m_{G/H}(\Gamma \backslash G/H)},    
\end{equation}
where $m_{G/H}(\Gamma\backslash G/H)$ denotes the volume of a fundamental domain for the $\Gamma$-action on $G/H$.
Indeed, $G/H$ is partitioned into $\Gamma$-translates of such a fundamental domain,
all of equal volume by $G$-invariance of $m_{G/H}$.
Each translate contains exactly one point of the orbit $\Gamma\cdot H$, 
so one expects $|\Gamma\cdot H\cap \cB_T|$ to be comparable to the number of translates that fit inside $\cB_T$, which is exactly the right-hand side of \eqref{eq:approx-orbit}.

The main subtlety is boundary behaviour.
The boundary of $\cB_T$ cuts through many translates of the fundamental domain, 
and in principle it could be biased toward avoiding (or capturing) orbit points. 
This matters if a neighbourhood of the boundary occupies a non-negligible proportion of $\cB_T$. 
To rule this out one assumes a regularity condition on the boundary,
usually formulated as a well-roundedness hypothesis. 
In our setting this condition will always hold, since $\cB_T$ is a sector of a norm ball (see, for example, \cite{sectors,MR3025156}).
The famous counting theorem of Eskin and McMullen \cite{MR1230290} shows that under such a hypothesis, 
mixing of the $G$-action makes the boundary contributions sufficiently ``random'',
and the heuristic \eqref{eq:approx-orbit} becomes an asymptotic formula.

The following statement is the specialization of \cite[Theorem 1.4]{MR1230290} to our setting.
\begin{theorem}
    \label{thm:EM-counting}
    As $T\to\infty$,
    \begin{equation}
        \label{eq:EM-counting}
        |\Gamma\cdot H \cap \cB_T| \sim \frac{m_H\bigl( (\Gamma \cap H) \backslash H\bigr)}{m_G(\Gamma \backslash G)} \, m_{G/H}(\cB_T).
    \end{equation}
\end{theorem}

Thus, in each of the three regimes appearing in Theorems \ref{thm:main} and \ref{thm:mainF} (definite, indefinite nonsplit, and split),
the remaining work is to compute the three terms on the right-hand side of \eqref{eq:EM-counting}.

\section{Definite case}
\label{sec:definite}

Throughout this section, $Q$ denotes an integral definite binary quadratic form of discriminant $4n<0$, and $H$ denotes its stabilizer in $G=\SL_2(\R)$.
Since $Q$ is definite, its stabilizer is conjugate to $\SO(2)$.
To simplify notation---and to avoid repeatedly writing conjugations---we will work with the standard rotation matrices model
\begin{equation}
    \label{eq:SO(2)-rotation-matrices}
    H=
    \Bigl\{
        \begin{pmatrix}
            \cos \theta & -\sin \theta \\
            \sin \theta & \cos \theta
        \end{pmatrix}
        \Bigm|
        \theta \in [0,2\pi)
    \Bigr\},
\end{equation}
keeping in mind that, for a general definite form $Q$, the actual stabilizer is a conjugate of this group.

\subsection*{Measures}

We first fix the three invariant measures that appear in \eqref{eq:EM-counting}.
View $G$ as the real algebraic group
\begin{equation}
    \label{eq:Galgebraic}
    G = \{ (a,b,c,d)\in\R^4 \mid ad - bc = 1\}.
\end{equation}
Let $\omega_G$ be the gauge form on $G$ which, on the chart $\{a\neq 0\}$, is given by
\begin{equation}
    \label{eq:omegaG}
    \omega_G = \frac{1}{a}\,da\wedge db\wedge dc.
\end{equation}
This form is left-invariant (see, e.g., \cite[Section~3.5]{PlatonovRapinchuk1994}),
and hence determines a Haar measure $m_G$ on $G$.
With this normalization one has (see, e.g., \cite[Sections~3.5 and~4.5]{PlatonovRapinchuk1994})
\begin{equation}
    \label{eq:GammaG}
    m_G(\Gamma \backslash G) = \frac{\pi^2}{6}.
\end{equation}
It remains to choose Haar measures $m_H$ on $H$ and $m_{G/H}$ on $G/H$ so that the disintegration identity \eqref{eq:Fubini} holds.

To ease the calculations (especially for $m_{G/H}$),
we use Iwasawa coordinates.
Every $g\in G$ can be written uniquely as
\[
    g=n(x)\,a(y)\,k(\theta),
    \qquad x\in\R,\ y>0,\ \theta\in[0,2\pi),
\]
where
\[
    n(x) =
    \begin{pmatrix}
        1 & x \\
        0 & 1
    \end{pmatrix},
    \quad
    a(y) = 
    \begin{pmatrix}
        y^{1/2} & 0 \\
        0 & y^{-1/2}
    \end{pmatrix},
    \quad
    k(\theta) =
    \begin{pmatrix}
        \cos \theta & -\sin \theta \\
        \sin \theta & \cos \theta
    \end{pmatrix}.
\]
A straightforward Jacobian computation gives
\[
    \omega_G = - \frac{1}{2y^2}\,dx\wedge dy\wedge d\theta.
\]
Accordingly, we take
\[
    m_{G/H}\bigl(n(dx)\,a(dy)\,H\bigr) = \frac{dx\,dy}{y^2},
    \qquad
    m_H\bigl(k(d\theta)\bigr) = \frac{d\theta}{2}.
\]
With this normalization,
\begin{equation}
    \label{eq:mH-norm-def}
    m_H(H) = 2\pi / 2 = \pi.
\end{equation}

\subsection*{The volume of $\cB_T$}

\begin{lemma}
    \label{l:size-B_T_definite}
    As $T\to\infty$,
    \[  
        m_{G/H}(\cB_T) = \frac{T}{|n|^{1/2}} + o(T).
    \]
\end{lemma}

\begin{proof}
    We may assume $Q$ is positive definite (otherwise replace it by $-Q$).
    Since only the $G$-orbit of $Q$ matters for this lemma, and since $Q$ is $G$-equivalent to $[|n|^{1/2},0,|n|^{1/2}]$, we may take
    \[
        Q=[|n|^{1/2},0,|n|^{1/2}].
    \]
    For this choice, the stabilizer is exactly the rotation matrices group \eqref{eq:SO(2)-rotation-matrices} (so no conjugation is needed).
    
    By definition,
    \[
        \cB_T = \{ n(x)a(y)H \in G/H \mid n(x)a(y)\cdot [|n|^{1/2},0,|n|^{1/2}] \in \cF_T\}.
    \]
    In matrix notation,
    \[
        n(x)a(y)\cdot [|n|^{1/2},0,|n|^{1/2}]
        = |n|^{1/2} y^{-1}
        \begin{pmatrix}
            x^2+y^2 & x \\
            x & 1
        \end{pmatrix},
    \]
    Set $M=T/|n|^{1/2}$. 
    Translating the conditions defining $\cF_T$ into inequalities in $(x,y)$ yields
    \[
        \cB_T
        = \bigl\{ n(x)a(y)H \in G/H \bigm| 
        \sqrt{\max(1-y^2,0)} \leq x \leq \sqrt{y(M - y)},
        1 / M \leq y \leq M
        \bigr\}.
    \]
    Therefore
    \begin{align*}
        m_{G/H}(\cB_T)
        &= \iint_{n(x)a(y)H \in \cB_T} \frac{dx\,dy}{y^2}
        = \int_{y=1/M}^M \frac{\sqrt{y(M - y)} - \sqrt{\max(1-y^2,0)}}{y^2} \, dy \\
        &= M \int_{u=1}^M \frac{\sqrt{u - u^2/M^2} - \sqrt{1-u^2/M^2}}{u^2} \, du
        + M \int_{u=M}^{M^2} \frac{\sqrt{u - u^2/M^2}}{u^2} \, du.
    \end{align*}
    For $u\in[1,M]$ the integrand admits the estimate
    \[
        \frac{\sqrt{u}-1}{u^2} + O(M^{-2}),
    \]
    so the first integral equals $1+O(M^{-1/2})$.
    For $u\in[M,M^2]$ the integrand is $O(u^{-3/2})$, hence the second integral is $O(M^{-1/2})$. 
    The claim follows easily.
\end{proof}

\subsection*{The covolume of $\Gamma\cap H$ in $H$}

\begin{lemma}
    \label{l:mHnorm}
    Write $Q=kQ'$, where $Q'$ is primitive, and let $d'$ be the discriminant of $Q'$ (so $d'=4n/k^2$).
    Then
    \[  
        m_{H}(\Gamma \cap H \backslash H) = \frac{\pi}{\omega(d')},
    \]
    where $\omega$ is as in Theorem \ref{thm:main}.
\end{lemma}

\begin{proof}
    Since $Q$ and $Q'$ have the same stabilizer, the normalization \eqref{eq:mH-norm-def} applies equally to both. The claim then follows from the classical fact (see \cite[Section~2.5.3]{MR2300780}) that $|\Gamma\cap H|=\omega(d')$.
\end{proof}

Combining Theorem \ref{thm:EM-counting}, Lemmas \ref{l:size-B_T_definite} and \ref{l:mHnorm}, and \eqref{eq:GammaG}, we obtain the desired asymptotic in the definite case.

\section{Indefinite nonsplit case}

Let $Q$ be an indefinite integral binary quadratic form of nonsquare discriminant $4n>0$, and let $H$ again denote its stabilizer in $G$.
We follow a similar strategy as in the definite case.

In this setting, $H$ is conjugate to $\SO(1,1)$.
As before, to keep notation light we fix a concrete model and work with it throughout, 
keeping in mind that one may have to conjugate to obtain the actual stabilizer of $Q$. 
Concretely, we take $H$ to be the subgroup of hyperbolic rotation matrices\footnote{We use the $s/2$ parametrization (instead of the more common $s$) to match the conventions in \cite{MR3219562}, which we will use in the next section.}
\begin{equation}
    \label{eq:SO(1,1)-hyperbolic}
    H =
    \Bigl\{
        \pm \begin{pmatrix}
            \cosh (s/2) & \sinh (s/2) \\
            \sinh (s/2) & \cosh (s/2)
        \end{pmatrix}
        \Bigm|
        s \in \R
    \Bigr\}.
\end{equation}

Every $g\in \SL_2(\R)$ admits a unique generalized Cartan decomposition (see \cite[Section 5]{MR3219562})
\begin{equation}
    \label{eq:cartan-decomp}
    g = \pm k(\theta) \, a(t) \, h(s)
\end{equation}
for a choice of sign $\pm$, with $\theta\in[0,2\pi)$ and $s,t\in\R$, where
\[
    k(\theta) = \begin{pmatrix} \cos (\theta/2) & -\sin(\theta/2)\\ \sin(\theta/2) & \cos(\theta/2)\end{pmatrix},
    \quad
    a(t) = \begin{pmatrix} e^{t/2} & 0\\ 0 & e^{-t/2}\end{pmatrix},
    \quad
    h(s) = \begin{pmatrix} \cosh (s/2) & \sinh (s/2)\\ \sinh (s/2) & \cosh (s/2)\end{pmatrix}.
\]
A straightforward Jacobian computation shows that the gauge form $\omega_G$ from \eqref{eq:omegaG}, when written in these coordinates, becomes
\[
    \omega_G = \frac{\cosh(t)}4 \,d\theta \wedge dt \wedge ds.
\]
Accordingly, we define invariant measures on $G/H$ and $H$ by
\begin{equation}
    \label{eq:cartan-measures}
    m_{G/H}\bigl( k(\theta)a(t)H\bigr) = \frac{\cosh(t)\,d\theta\,dt}4,
    \qquad
    m_H\bigl(\pm h(ds)\bigr) = ds.
\end{equation}

\subsection*{The volume of $\cB_T$}

\begin{lemma}
    \label{l:kugla-indefinite}
    As $T \to \infty$,
    \[  
        m_{G/H}(\cB_T) = \frac{T}{2n^{1/2}} + o(T).
    \]
\end{lemma}

\begin{proof}
    The argument parallels Lemma~\ref{l:size-B_T_definite}, so we spell out only the needed steps.

    Since only the $G$-orbit matters here, we may assume
    \[
        Q = [n^{1/2},0,-n^{1/2}]
    \]
    in which case the stabilizer is exactly the group of hyperbolic rotation matrices \eqref{eq:SO(1,1)-hyperbolic}.
    Thus
    \[
        \cB_T = \{ k(\theta)a(t)H \in G/H \mid k(\theta)\,a(t)\cdot [n^{1/2},0,-n^{1/2}] \in \cF_T\}.
    \]
    A direct computation gives
    \[
        k(\theta)a(t)\cdot [n^{1/2},0,-n^{1/2}]
        = n^{1/2}
        \begin{pmatrix}
            \sinh t+\cosh t\,\cos\theta & \cosh t\,\sin\theta\\
            \cosh t\,\sin\theta & \sinh t-\cosh t\,\cos\theta
        \end{pmatrix}.
    \]
    Writing $M=T/n^{1/2}$, the conditions defining $\cF_T$ translate into
    \begin{align*}
        \cB_T
        &= \bigl\{ k(\theta)a(t)H \in G/H \bigm| 
        |\sinh t+\cosh t\,\cos\theta| \leq M,\ 
        |\sinh t-\cosh t\,\cos\theta| \leq M,\
        \theta \in  [0,\pi/2)
        \bigr\}.
    \end{align*}
    (The restriction $\theta\in[0,\pi/2)$ comes from $a>c$ and $b\geq 0$.)
    
    For $\theta\in[0,\pi/2)$ the two inequalities in $t$ are equivalent to
    \[
        |\sinh t| + \cosh t\,\cos\theta \leq M.
    \]
    Let $u=\sinh t$. Since $du=\cosh t\,dt$, we obtain, by symmetry in $u$,
    \[
        m_{G/H}(\cB_T)
        = \iint_{k(\theta)a(t)H \in \cB_T} \frac{\cosh t\,d\theta\,dt}{4}
        = \frac12 \int_{\theta=0}^{\pi/2} \int_{\substack{u\geq 0 \\ u+\cos \theta \sqrt{1+u^2} \leq M}} du\,d\theta.
    \]
    Only $u \in [0,M]$ contribute.
    We split this interval depending on whether the inequality $u+\cos\theta\sqrt{1+u^2}\leq M$ forces a nontrivial lower bound on $\theta$. 
    The transition point is $u_0=(M^2-1)/(2M)$, and using $\arcsin=\pi/2-\arccos$ we get
    \[
        m_{G/H}(\cB_T)
        = \frac12 \, \Bigl(
            \int_{u=0}^{u_0} \frac\pi2 \,du
            + \int_{u_0}^M \arcsin \frac{M-u}{\sqrt{1+u^2}} \,du
            \Bigr).
    \]
    With the substitution $u=Mv$, dominated convergence gives, as $M\to\infty$,
    \[
        \frac{m_{G/H}(\cB_T)}{M/2}
        \sim 
        \int_{v=0}^{1/2} \frac\pi2 \,dv
            + \int_{1/2}^1 \arcsin \frac{1-v}{v} \,dv.
    \]
    A routine integration by parts followed by a trigonometric substitution shows that the right-hand side equals $1$, which implies the claimed asymptotic.
\end{proof}

\subsection*{The covolume of $\Gamma\cap H$ in $H$}

It remains to compute the factor $m_H\bigl((\Gamma\cap H)\backslash H\bigr)$.
Recall that our Haar measure on $H$ is induced by the form
\begin{equation}
    \label{eq:omega-SO(1,1)-hyperbolic}
    \omega_H(h) = ds
    \qquad
    \text{at }
    h(s) = \pm \begin{pmatrix}
        \cosh (s/2) & \sinh (s/2) \\
        \sinh (s/2) & \cosh (s/2)
    \end{pmatrix},
    \quad
    s \in \R.
\end{equation}
We will also express $\omega_H$ in diagonal coordinates:
\begin{equation}
    \label{eq:omega-SO(1,1)-diagonal}
    \omega_H(\tilde h(t)) = \frac{2 dt}t
    \qquad
    \text{at }
    \tilde h(t) = P\begin{pmatrix}
        t & 0 \\
        0 & t^{-1}
    \end{pmatrix}
    P^{-1},
    \quad
    t \neq 0
\end{equation}
where $P = \frac1{\sqrt2}(\begin{smallmatrix} 1 & 1 \\ 1 & -1 \end{smallmatrix})$.
In these coordinates, $m_H(\tilde h(dt))=2\,dt/|t|$. 

To compute $m_H\bigl((\Gamma\cap H)\backslash H\bigr)$ we use the following classical description of stabilizers (see \cite[Section~6.12]{MR2300780}).

\begin{lemma}
    \label{lem:stabilizer_indefinite}
    Let $[a,b,c]$ be a primitive indefinite quadratic form of nonsquare discriminant $d>0$.
    Its $\SL_2(\Z)$-stabilizer is isomorphic to $\Z/2\Z \times \Z$, generated by $-I$ and the infinite-order matrix
    \[
        T_0
        =
        \begin{pmatrix}
            \frac{t - b s}{2} & -c s \\
            a s & \frac{t + b s}{2}
        \end{pmatrix},
    \]
    where $(t,s)$ is the minimal positive solution of the Pell-type equation
    \[
        t^2 - d s^2 = 4.
    \]
\end{lemma}

\begin{remark}
    Let $\cO_d = \Z\Bigl[\dfrac{d+\sqrt{d}}{2}\Bigr]$ be the quadratic order of discriminant $d$ in $\Q(\sqrt{d})$.
    It is classical that 
    \[
        \frac{t + s\sqrt{d}}{2},
    \]
    with $t,s$ as in Lemma~\ref{lem:stabilizer_indefinite}, is either a fundamental unit (if $\kappa(d)=2$) or the square of a fundamental unit (if $\kappa(d)=1$) of $\mathcal{O}_d$.
    Consequently, the eigenvalues of $T_0$ are
    \[
        \lambda = \varepsilon_d^{\,2/\kappa(d)}
        \qquad\text{and}\qquad
        \lambda^{-1}.
    \]
\end{remark}

\begin{lemma}
    \label{l:stab-indef}
    We have
    \[
        m_H \bigl((\Gamma \cap H)\backslash H\bigr)
        =
        2 \log \varepsilon_d^{\,2/\kappa(d)}.
    \] 
\end{lemma}

\begin{proof}
    Let $H^0$ be the identity component of $H$. Then $H=\pm H^0$, and since $-I\in \Gamma\cap H$,
    \[
        m_H \bigl((\Gamma \cap H)\backslash H\bigr)
        = m_H \bigl((\Gamma \cap H^0)\backslash H^0\bigr).
    \]
    By Lemma~\ref{lem:stabilizer_indefinite}, the group $\Gamma\cap H^0$ is generated by the hyperbolic element $T_0$.

    Using the diagonal coordinate from \eqref{eq:omega-SO(1,1)-diagonal}, 
    we have\footnote{Here we conjugate by $P$ to match our fixed model \eqref{eq:SO(1,1)-hyperbolic}.
    This avoids introducing an additional conjugation coming from the stabilizer of the general $Q$.}
    \[
        H^0 =
        \Bigl\{
            P
            \begin{pmatrix}
                t & 0 \\
                0 & t^{-1}
            \end{pmatrix}
            P^{-1}
            \Bigm|
            t>0
        \Bigr\},
    \]
    and we may identify $T_0$ with
    \[
        P \begin{pmatrix}
            \lambda & 0 \\
            0 & \lambda^{-1}
        \end{pmatrix} P^{-1},
        \qquad
        \lambda=\varepsilon_d^{\,2/\kappa(d)}>1.
    \]

    A fundamental domain for $(\Gamma\cap H^0)\backslash H^0$ is therefore given by $t\in[1,\lambda]$, and hence
    \[
        m_H \bigl((\Gamma \cap H^0)\backslash H^0\bigr)
        = \int_{1}^{\lambda} \frac{2 dt}{t}
        = 2 \log \lambda,
    \]
    as claimed.
\end{proof}

Combining Theorem~\ref{thm:EM-counting}, Lemmas~\ref{l:kugla-indefinite} and~\ref{l:stab-indef}, and \eqref{eq:GammaG}, we obtain the desired asymptotic for an indefinite nonsplit form $Q$.

\section{Split case}
\label{sec:split}

Finally, we treat the remaining case in Theorem~\ref{thm:main}, namely indefinite \emph{split} forms.
Throughout this section let $n$ be a positive square, and let $Q$ be an integral binary quadratic form of discriminant $4n$.

As before, our goal is to estimate $|(\Gamma\cdot Q)\cap \cF_T|$ for $\cF_T$ as in \eqref{eq:cFT}.
The stabilizer $H$ of $Q$ in $G$ is again conjugate to $\SO(1,1)$.
However, in the split case one has $\Gamma\cap H=\{\pm I\}$, so $\Gamma\cap H$ is not a lattice in the noncompact group $H$.
Consequently, the Eskin--McMullen counting theorem does not apply directly.

Fortunately, our situation fits into the framework developed in a very nice paper by Oh and Shah~\cite{MR3219562}.
Before using their result, we recall the generalized Cartan decomposition \eqref{eq:cartan-decomp} and the associated measures \eqref{eq:cartan-measures}.

After adapting to our setting and normalizations,%
%\footnote{In our convention, the Haar measure on $G$ has density $\cosh(t)/4$ in generalized Cartan coordinates, whereas \cite{MR3219562} uses a different normalization. This only changes the final constant by a fixed factor.}
\footnote{In particular, our normalization of the measure on $G$ has density $\cosh(t)/4$ in generalized Cartan coordinates,
whereas Oh and Shah use $\cosh(t)$.
In fact, they write $\sinh(t)$ which is presumably a typo -- in any case, this is inconsequential for their work since both functions are $\sim e^t/2$ as $t \to \infty$.}
\cite[Theorem 6.1]{MR3219562} yields
\begin{equation}
    \label{eq:oh-shah}
    \bigl|(\Gamma\cdot Q) \cap \cF_T\bigr|
    \sim \frac{T \log T \int_\Theta \|k(\theta)\cdot Q_1\|^{-1}\,d\theta}{2 m_G(\Gamma \backslash G)},
\end{equation}
where $Q_1$ is the highest-weight component of $Q$ for the $a(t)$-action (defined below), and
\[
    \Theta
    =  \{\theta\in[0,2\pi)\mid k(\theta)\cdot Q_1 \in \R_{\geq 0}\,\cF_1\}.
\]
Our first step is to relate the integral appearing in \eqref{eq:oh-shah} to the volume $m_{G/H}(\cB_T)$,
a fact clear from the proofs of Oh and Shah but not explicitly stated in \cite[Section 6]{MR3219562}.

As in the proof of Lemma~\ref{l:kugla-indefinite}, we may assume (since only the $G$-orbit matters) that
\[
    Q=[n^{1/2},0,-n^{1/2}].
\]

Consider the norm on $V$ given by
\[
    \|[a,b,c]\| = \max\bigl\{|a|,\tfrac13|b|,|c|\bigr\}.
\]
For forms of discriminant $4n$ and large enough $T$, the condition $|a|,|c|\leq T$ is equivalent to $\|[a,b,c]\|\leq T$.

The space $V$ decomposes into $a(t)$-weight spaces:
for each $\lambda\in\Lambda:=\{-1,0,1\}$ there is a subspace $V_\lambda$ such that
\[
    V=\bigoplus_{\lambda\in\Lambda} V_\lambda
    \quad\text{and $V_\lambda$ is $e^{\lambda t}$-eigenspace of $a(t)$ for all $t$}.
\]
Concretely,
\[
    V_1 = \R[1,0,0],
    \quad
    V_0 = \R[0,1,0],
    \quad
    V_{-1} = \R[0,0,1].
\]
Let $Q_1$ be the $V_1$-component (i.e., highest weight) of $Q$. Then, as $t\to\infty$,
\[
    a(t) \cdot  Q = (1+o(1)) \, e^{t} \, Q_1,
\]
and hence
\[
    k(\theta) a(t) \cdot  Q = (1+o(1)) \, e^{t}\, k(\theta) \cdot  Q_1,
\]
where the error term is uniform in $\theta$ by compactness of $\{k(\theta) \mid \theta \in [0,2\pi]\}$.
In particular, for each fixed $\theta$,
\begin{equation}
    \label{eq:ttheta}
    \|k(\theta) a(t) \cdot  Q\|
    \sim e^t\,\|k(\theta)\cdot  Q_1\|.
\end{equation}

\begin{proposition}
    \label{prop:integral-volume}
    As $T\to\infty$,
    \[
        m_{G/H}(\cB_T)
        \sim
        \frac{T}4 \int_{\Theta}\|k(\theta)\cdot Q_1\|^{-1}\,d\theta.
    \]
\end{proposition}

\begin{proof}
    [Sketch proof]
    Variants of this statement appear in the literature (for instance in much greater generality in \cite{MR2488484}), 
    so we only indicate the main idea.
    We focus on the contribution from the region $t\geq 0$; the region $t<0$ is analogous.

    For large $t$, the direction of $k(\theta)a(t)\cdot Q$ is governed by $k(\theta)\cdot Q_1$, in the sense that
    \[
        \frac{k(\theta)a(t)\cdot Q}{\|k(\theta)a(t)\cdot Q\|}
        \to
        \frac{k(\theta)\cdot Q_1}{\|k(\theta)\cdot Q_1\|}.
    \]
    Therefore only $\theta\in\Theta$ contribute to the main term, and for each such $\theta$ we need to integrate over those $t\geq 0$ for which $\|k(\theta)a(t)\cdot Q\|\leq T$.
    Writing
    \[
        t_T(\theta):=\sup\{t\geq 0 \mid \|k(\theta)a(t)\cdot Q\|\leq T\},
    \]
    the asymptotic \eqref{eq:ttheta} implies
    \[
        e^{t_T(\theta)}\sim \frac{T}{\|k(\theta)\cdot Q_1\|}.
    \]
    Using the density $\cosh(t)/4$ from \eqref{eq:cartan-measures}, we obtain
    \[
        m_{G/H}(\cB_T)
        \sim \frac14 \int_{\theta\in\Theta} \int_{t=0}^{t_T(\theta)} \cosh(t)\,dt\,d\theta
        \sim \frac{T}8 \int_{\theta\in\Theta} \|k(\theta)\cdot Q_1\|^{-1} \,d\theta.
    \]
    The contribution from $t<0$ is the same, giving the stated factor $T/4$.
\end{proof}

Combining Proposition~\ref{prop:integral-volume} with \eqref{eq:oh-shah} yields
\[
    \bigl|(\Gamma\cdot Q) \cap \cF_T\bigr|
    \sim \frac{2 \log T \cdot m_{G/H}(\cB_T)}{m_G(\Gamma \backslash G)}.
\]
Invoking Lemma~\ref{l:kugla-indefinite} and \eqref{eq:GammaG} now gives the asymptotic in the split case, completing the proof of Theorem~\ref{thm:mainF}.

\section*{Acknowledgments}
G.D.\ and M.K.\ were supported by the Croatian Science Foundation under the project no.\ IP-2022-10-5008 (TEBAG).
R.M.\ was supported by the Croatian Science Foundation under the project no.\ HRZZ-IP-2022-10-5116 (FANAP).
M.K.\ acknowledges support from the project “Implementation of cutting-edge research and its application as part of the Scientific Center of Excellence for Quantum and Complex Systems, and Representations of Lie Algebras”, Grant No.\ PK.1.1.10.0004, co-financed by the European Union through the European Regional Development Fund -- Competitiveness and Cohesion Programme 2021-2027.

\printbibliography

@book {MR2300780,
    AUTHOR = {Buchmann, Johannes and Vollmer, Ulrich},
     TITLE = {Binary quadratic forms},
    SERIES = {Algorithms and Computation in Mathematics},
    VOLUME = {20},
      NOTE = {An algorithmic approach},
 PUBLISHER = {Springer, Berlin},
      YEAR = {2007},
     PAGES = {xiv+318},
      ISBN = {978-3-540-46367-2; 3-540-46367-4},
   MRCLASS = {11E16 (11-01 11R11 11Y40)},
  MRNUMBER = {2300780},
MRREVIEWER = {Duncan\ A.\ Buell},
}

@book {split-zagier,
    AUTHOR = {Zagier, D. B.},
     TITLE = {Zetafunktionen und quadratische {K}\"orper},
    SERIES = {Hochschultext. [University Textbooks]},
      NOTE = {Eine Einf\"uhrung in die h\"ohere Zahlentheorie. [An
              introduction to higher number theory]},
 PUBLISHER = {Springer-Verlag, Berlin-New York},
      YEAR = {1981},
     PAGES = {viii+144},
      ISBN = {3-540-10603-0},
   MRCLASS = {10-01},
  MRNUMBER = {631688},
MRREVIEWER = {Matti\ Jutila},
}

@article {MR3025156,
    AUTHOR = {Benoist, Yves and Oh, Hee},
     TITLE = {Effective equidistribution of {$S$}-integral points on
              symmetric varieties},
   JOURNAL = {Ann. Inst. Fourier (Grenoble)},
  FJOURNAL = {Universit\'e{} de Grenoble. Annales de l'Institut Fourier},
    VOLUME = {62},
      YEAR = {2012},
    NUMBER = {5},
     PAGES = {1889--1942},
      ISSN = {0373-0956,1777-5310},
   MRCLASS = {11G35},
  MRNUMBER = {3025156},
MRREVIEWER = {Konstantinos\ Draziotis},
       DOI = {10.5802/aif.2738},
       URL = {https://doi-org.ezproxy.nsk.hr/10.5802/aif.2738},
}

@article {MR2488484,
    AUTHOR = {Gorodnik, Alexander and Oh, Hee and Shah, Nimish},
     TITLE = {Integral points on symmetric varieties and {S}atake
              compactifications},
   JOURNAL = {Amer. J. Math.},
  FJOURNAL = {American Journal of Mathematics},
    VOLUME = {131},
      YEAR = {2009},
    NUMBER = {1},
     PAGES = {1--57},
      ISSN = {0002-9327,1080-6377},
   MRCLASS = {22E40 (14L30)},
  MRNUMBER = {2488484},
MRREVIEWER = {B.\ Sury},
       DOI = {10.1353/ajm.0.0034},
       URL = {https://doi-org.ezproxy.nsk.hr/10.1353/ajm.0.0034},
}

@article {sectors,
    AUTHOR = {Gorodnik, Alexander and Oh, Hee and Shah, Nimish},
     TITLE = {Strong wavefront lemma and counting lattice points in sectors},
   JOURNAL = {Israel J. Math.},
  FJOURNAL = {Israel Journal of Mathematics},
    VOLUME = {176},
      YEAR = {2010},
     PAGES = {419--444},
      ISSN = {0021-2172,1565-8511},
   MRCLASS = {22E30 (11E08)},
  MRNUMBER = {2653201},
MRREVIEWER = {B.\ Sury},
       DOI = {10.1007/s11856-010-0035-8},
       URL = {https://doi-org.ezproxy.nsk.hr/10.1007/s11856-010-0035-8},
}

@book {MR2978290,
    AUTHOR = {Horn, Roger A. and Johnson, Charles R.},
     TITLE = {Matrix analysis},
   EDITION = {Second},
 PUBLISHER = {Cambridge University Press, Cambridge},
      YEAR = {2013},
     PAGES = {xviii+643},
      ISBN = {978-0-521-54823-6},
   MRCLASS = {15-01},
  MRNUMBER = {2978290},
MRREVIEWER = {Mohammad\ Sal\ Moslehian},
}

@article {MR3219562,
    AUTHOR = {Oh, Hee and Shah, Nimish A.},
     TITLE = {Limits of translates of divergent geodesics and integral
              points on one-sheeted hyperboloids},
   JOURNAL = {Israel J. Math.},
  FJOURNAL = {Israel Journal of Mathematics},
    VOLUME = {199},
      YEAR = {2014},
    NUMBER = {2},
     PAGES = {915--931},
      ISSN = {0021-2172,1565-8511},
   MRCLASS = {22E40 (11P21 53C22 53D25)},
  MRNUMBER = {3219562},
MRREVIEWER = {S.\ G.\ Dani},
       DOI = {10.1007/s11856-013-0063-2},
       URL = {https://doi-org.ezproxy.nsk.hr/10.1007/s11856-013-0063-2},
}

@article {MR1230290,
    AUTHOR = {Eskin, Alex and McMullen, Curt},
     TITLE = {Mixing, counting, and equidistribution in {L}ie groups},
   JOURNAL = {Duke Math. J.},
  FJOURNAL = {Duke Mathematical Journal},
    VOLUME = {71},
      YEAR = {1993},
    NUMBER = {1},
     PAGES = {181--209},
      ISSN = {0012-7094,1547-7398},
   MRCLASS = {22E40 (57S30 58F17)},
  MRNUMBER = {1230290},
MRREVIEWER = {Nimish\ A.\ Shah},
       DOI = {10.1215/S0012-7094-93-07108-6},
       URL = {https://doi.org/10.1215/S0012-7094-93-07108-6},
}

@article{ADDP,
  author  = {Nikola Adžaga and Goran Dražić and Andrej Dujella and Attila Pethő},
  title   = {Asymptotics of {D}(q)-pairs and triples via {L}-functions of {D}irichlet characters},
  journal = {The Ramanujan Journal},
  volume  = {66},
  number  = {1},
  pages   = {1--23},
  year    = {2025},
  doi     = {10.1007/s11139-024-00979-3},
  url     = {https://doi.org/10.1007/s11139-024-00979-3}
}

@article{ADKT,
  author = {Nikola Adžaga and Andrej Dujella and Dario Kreso and Petra Tadić},
  title = {Triples which are {D}(n)-sets for several n's},
  journal = {J. Number Theory},
  volume = {184},
  year = {2018},
  pages = {330--341},
  doi = {10.1016/j.jnt.2017.08.013}
}

@article{Stoll,
  author = {Michael Stoll},
  title = {Diagonal genus 5 curves, elliptic curves over $\mathbb{Q}(t)$, and rational {D}iophantine quintuples},
  journal = {Acta Arithmetica},
  volume = {190},
  year = {2019},
  pages = {239--261},
  doi = {10.4064/aa180416-4-10}
}

@article{Badesa,
  author  = {Javier Badesa},
  title   = {On the asymptotics of {D(n)}‑pairs and triples},
  journal = {The Ramanujan Journal},
  volume  = {67},
  number  = {101},
  pages   = {–}, 
  year    = {2025},
  doi     = {10.1007/s11139-025-01149-9},
  url     = {https://doi.org/10.1007/s11139-025-01149-9}
}

@article{B-D,
  author = {Alan Baker and Harold Davenport},
  title = {The equations $3x^2 - 2 = y^2$ and $8x^2 - 7 = z^2$},
  journal = {Quart. J. Math. Oxford Ser. (2)},
  volume = {20},
  year = {1969},
  pages = {129--137},
  doi = {10.1093/qmath/20.1.129}
}

@article{HTZ,
  author = {Bo He and Alain Togbé and Volker Ziegler},
  title = {There is no {D}iophantine quintuple},
  journal = {Trans. Amer. Math. Soc.},
  volume = {371},
  year = {2019},
  pages = {6665--6709},
  doi = {10.1090/tran/7573}
}

@article{Gibbs1,
  author = {Philip Gibbs},
  title = {Some rational {D}iophantine sextuples},
  journal = {Glas. Mat. Ser. III},
  volume = {41},
  year = {2006},
  pages = {195--203},
  doi = {10.3336/gm.41.2.02}
}

@article{DKMS,
  author = {Andrej Dujella and Matija Kazalicki and Miljen Mikić and Márton Szikszai},
  title = {There are infinitely many rational {D}iophantine sextuples},
  journal = {Int. Math. Res. Not. IMRN},
  year = {2017},
  volume = {2017},
  number = {2},
  pages = {490--508},
  doi = {10.1093/imrn/rnv376}
}

@article{Duje-asymptotic,
  author  = {Andrej Dujella},
  title   = {On the number of {D}iophantine $m$-tuples},
  journal = {The Ramanujan Journal},
  volume  = {15},
  number  = {1},
  pages   = {37--46},
  year    = {2008},
  doi     = {10.1007/s11139-007-9066-0},
  url     = {https://doi.org/10.1007/s11139-007-9066-0}
}

@article{duje-crelle,
  author  = {Andrej Dujella},
  title   = {There are only finitely many {D}iophantine quintuples},
  journal = {Journal für die reine und angewandte Mathematik},
  volume  = {566},
  year    = {2004},
  pages   = {183--214},
  doi     = {10.1515/crll.2004.037},
  url     = {https://doi.org/10.1515/crll.2004.037}
}

@incollection{Duje-Matija,
  author = {Andrej Dujella and Matija Kazalicki},
  title = {More on {D}iophantine sextuples},
  booktitle = {Number Theory – Diophantine Problems, Uniform Distribution and Applications},
  publisher = {Springer},
  year = {2017},
  pages = {227--235},
  doi = {10.1007/978-3-319-55357-3_14}
}

@article{DujellaPetricEvic2008,
  author  = {Andrej Dujella and Vinko Petričević},
  title   = {Strong {D}iophantine Triples},
  journal = {Experimental Mathematics},
  volume  = {17},
  number  = {1},
  pages   = {83--89},
  year    = {2008},
  doi     = {10.1080/10586458.2008.10129020},
  url     = {https://doi.org/10.1080/10586458.2008.10129020}
}

@article{DKP-sext,
  author = {Andrej Dujella and Matija Kazalicki and Vinko Petričević},
  title = {There are infinitely many rational {D}iophantine sextuples with square denominators},
  journal = {J. Number Theory},
  volume = {205},
  year = {2019},
  pages = {340--346},
  doi = {10.1016/j.jnt.2019.04.003}
}

@article{DKP-square,
  author  = {Andrej Dujella and Matija Kazalicki and Vinko Petričević},
  title   = {D(n)-quintuples with square elements},
  journal = {Revista de la Real Academia de Ciencias Exactas, Físicas y Naturales. Serie A. Matemáticas},
  volume  = {115},
  year    = {2021},
  pages   = {Article 172},
  doi     = {10.1007/s13398-021-01115-2},
  url     = {https://doi.org/10.1007/s13398-021-01115-2}
}

@article{DKP-reg,
  author = {Andrej Dujella and Matija Kazalicki and Vinko Petričević},
  title = {Rational {D}iophantine sextuples containing two regular quadruples and one regular quintuple},
  journal = {Acta Math. Spalat.},
  volume = {1},
  year = {2020},
  pages = {19--27},
  doi = {10.32817/ams.1.1.2}
}

@article{DKP-strong,
  author  = {Andrej Dujella and Matija Kazalicki and Vinko Petričević},
  title   = {Rational {D}iophantine sextuples with strong pair},
  journal = {Revista de la Real Academia de Ciencias Exactas, Físicas y Naturales. Serie A. Matemáticas},
  volume  = {119},
  year    = {2025},
  pages   = {Article 36},
  doi     = {10.1007/s13398-024-01103-3},
  url     = {https://doi.org/10.1007/s13398-024-01103-3}
}

@article{DK-finite,
  author = {Andrej Dujella and Matija Kazalicki},
  title = {Diophantine m-tuples in finite fields and modular forms},
  journal = {Res. Number Theory},
  volume = {7},
  year = {2021},
  pages = {Article 3},
  doi = {10.1007/s40993-020-00213-3}
}

@article{DrazicKazalicki,
  author       = {Goran Dražić and Matija Kazalicki},
  title        = {Rational {D}(q)-quadruples},
  journal      = {Indagationes Mathematicae},
  volume       = {33},
  number       = {2},
  pages        = {440--449},
  year         = {2022},
  doi          = {10.1016/j.indag.2021.09.009},
  url          = {https://doi.org/10.1016/j.indag.2021.09.009}
}

@article{Drazic,
  author       = {Goran Dražić},
  title        = {Rational {D}(q)-quintuples},
  journal      = {Revista de la Real Academia de Ciencias Exactas, Físicas y Naturales. Serie A. Matemáticas},
  volume       = {116},
  number       = {1},
  pages        = {9},
  year         = {2022},
  doi          = {10.1007/s13398-021-01146-9},
  url          = {https://doi.org/10.1007/s13398-021-01146-9}
}

@article{DP,
  author  = {Andrej Dujella and Vinko Petričević},
  title   = {Diophantine quadruples with the properties $D(n_1)$ and $D(n_2)$},
  journal = {Revista de la Real Academia de Ciencias Exactas, Físicas y Naturales. Serie A. Matemáticas (RACSAM)},
  volume  = {114},
  year    = {2020},
  pages   = {448--456},
  doi     = {10.1007/s13398-019-00747-9},
  url     = {https://doi.org/10.1007/s13398-019-00747-9}
}

@article{DP2,
  author  = {Andrej Dujella and Vinko Petričević},
  title   = {Doubly regular {D}iophantine quadruples},
  journal = {Revista de la Real Academia de Ciencias Exactas, Físicas y Naturales. Serie A. Matemáticas},
  volume  = {114},
  year    = {2020},
  pages   = {Article 189},
  doi     = {10.1007/s13398-020-00921-4},
  url     = {https://doi.org/10.1007/s13398-020-00921-4}
}

@book{duje-book,
  author    = {Andrej Dujella},
  title     = {Diophantine \textit{m}-tuples and Elliptic Curves},
  series    = {Developments in Mathematics},
  volume    = {79},
  publisher = {Springer},
  year      = {2024},
  address   = {Cham},
  doi       = {10.1007/978-3-031-56724-7},
  url       = {https://doi.org/10.1007/978-3-031-56724-7}
}

@article{Duje-Notices,
  author  = {Andrej Dujella},
  title   = {What is ... a {D}iophantine $m$-tuple?},
  journal = {Notices of the American Mathematical Society},
  volume  = {63},
  year    = {2016},
  number  = {7},
  pages   = {772--774},
  url     = {https://www.ams.org/notices/201607/rnoti-p772.pdf}
}

@book{PlatonovRapinchuk1994,
    AUTHOR = {Platonov, Vladimir and Rapinchuk, Andrei},
     TITLE = {Algebraic groups and number theory},
    SERIES = {Pure and Applied Mathematics},
    VOLUME = {139},
      NOTE = {Translated from the 1991 Russian original by Rachel Rowen},
 PUBLISHER = {Academic Press, Inc., Boston, MA},
      YEAR = {1994},
     PAGES = {xii+614},
      ISBN = {0-12-558180-7},
   MRCLASS = {11E57 (11-02 20Gxx)},
  MRNUMBER = {1278263},
}

\end{document}